\documentclass{article}
\usepackage{mathrsfs, amsmath, amsthm}
\usepackage{graphicx, xcolor}
\usepackage{tkz-graph, tkz-berge}
\usepackage{caption, subcaption}
\usepackage{array}
\usepackage{cases}
\makeatletter
\def\th@plain{%
  \upshape 
}
\makeatother

\makeatletter
\renewenvironment{proof}[1][\proofname]{\par
  \pushQED{\qed}%
  \normalfont \topsep6\p@\@plus6\p@\relax
  \trivlist
  \item[\hskip\labelsep
        \bfseries
    #1\@addpunct{.}]\ignorespaces
}{%
  \popQED\endtrivlist\@endpefalse
}
\makeatother

\newtheorem{theorem}{Theorem}[section]
\newtheorem{lemma}{Lemma}
\newtheorem{corollary}{Corollary}

\newtheorem{conjecture}{Conjecture}
\newtheorem*{conjecture*}{Conjecture}
\newtheorem{case}{Case}
\newtheorem{subcase}{Subcase}[case]

\theoremstyle{definition}

\newtheorem{remark}{Remark}

\usepackage[left=25mm,top=25mm,bottom=25mm,right=25mm]{geometry}
\usepackage[T1]{fontenc} 
\usepackage{newtxtext}
\usepackage[varvw]{newtxmath}

\usepackage{paralist, tablists}
\usepackage[inline]{enumitem}
\usepackage[square, numbers, sort&compress]{natbib}

\usepackage[pdftex,%
  bookmarks=true,%
  bookmarksnumbered=true, 
  bookmarksopen=true, 
  plainpages=false,%
  pdfpagelabels,%
  colorlinks=true, 
  linkcolor=blue, 
  citecolor=blue,%
  anchorcolor=green,
  urlcolor= blue,
  breaklinks=true,
  hyperindex=true]{hyperref}

\newcounter{Hcase}
\newcounter{Hclaim}

\newcommand{\resetcounter}{\stepcounter{Hcase}\setcounter{case}{0}\stepcounter{Hclaim}\setcounter{claim}{0}}

\newcommand{\etal}{et~al.\ }

\def\int(#1){\mathrm{int}(#1)}
\def\ext(#1){\mathrm{ext}(#1)}
\def\Int(#1){\mathrm{Int}(#1)}
\def\Ext(#1){\mathrm{Ext}(#1)}
\def\ad(#1){\mathrm{ad}(#1)}
\def\mad(#1){\mathrm{mad}(#1)}
\def\la(#1){\mathrm{la}(#1)}

\newcommand{\oddgirth}{\mathrm{oddgirth}}
\newcommand{\evengirth}{\mathrm{evengirth}}
\begin{document}
\title{Odd graph and its applications to the strong edge coloring}
\author{Tao Wang\footnote{{\tt Corresponding
author: wangtao@henu.edu.cn} }{ } \quad Xiaodan Zhao\\
{\small Institute of Applied Mathematics}\\
{\small Henan University, Kaifeng, 475004, P. R. China}}
\maketitle
\begin{abstract}
A strong edge coloring of a graph is a proper edge coloring in which every color class is an induced matching. The strong chromatic index $\chiup_{s}'(G)$ of a graph $G$ is the minimum number of colors in a strong edge coloring of $G$. Let $\Delta \geq 4$ be an integer. In this note, we study the odd graphs and show the existence of some special walks. By using these results and Chang's ideas in [Discuss. Math. Graph Theory 34 (4) (2014) 723--733], we show that every planar graph with maximum degree at most $\Delta$ and girth at least $10 \Delta - 4$ has a strong edge coloring with $2\Delta - 1$ colors. In addition, we prove that if $G$ is a graph with girth at least $2\Delta - 1$ and $\mad(G) < 2 + \frac{1}{3\Delta - 2}$, where $\Delta$ is the maximum degree and $\Delta \geq 4$, then $\chiup_{s}'(G) \leq 2\Delta - 1$; if $G$ is a subcubic graph with girth at least $8$ and $\mad(G) < 2 + \frac{2}{23}$, then $\chiup_{s}'(G) \leq 5$. 
\end{abstract}

\section{Introduction}
All graphs considered in this paper are simple, finite and undirected. A {\em strong edge coloring} of a graph is a proper edge coloring in which every color class is an induced matching. The {\em strong chromatic index} $\chiup_{s}'(G)$ of a graph $G$ is the minimum number of colors in a strong edge coloring of $G$. Fouquet and Jolivet \cite{MR737086, MR776805} introduced the notion of strong edge coloring for the radio networks and their frequencies assignment problem. 

The greedy algorithm provides an upper bound $2\Delta(\Delta - 1) + 1$ on the strong chromatic index, where $\Delta$ is the maximum degree of the graph. In 1985, Erd{\H{o}}s and Ne{\v s}et{\v r}il constructed graphs with strong chromatic index $\frac{5}{4}\Delta^{2}$ when $\Delta$ is even, $\frac{1}{4}(5\Delta^{2} - 2 \Delta + 1)$ when $\Delta$ is odd. Inspired by their construction, they proposed the following strong edge coloring conjecture during a seminar in Prague.
\begin{conjecture}[Erd{\H{o}}s and Ne{\v s}et{\v r}il, see \cite{MR975526}]\label{C}%
If $G$ is a graph with maximum degree $\Delta$, then
\[
\chiup_{s}'(G) \leq
\begin{cases}
\frac{5}{4} \Delta^{2}, & \text{if $\Delta$ is even}, \\
\frac{1}{4}(5\Delta^{2} - 2 \Delta + 1), & \text{if $\Delta$ is odd.}
\end{cases}
\]
\end{conjecture}

For $\Delta = 3$, the conjecture was confirmed independently by Hor\'{a}k \etal \cite{MR1217390} and Andersen \cite{MR1189847}. For $\Delta = 4$, Cranston \cite{MR2264374} showed that $22$ suffices, and this was further improved to $21$ by Huang, Santana and Yu \cite{Huang}. In 1997, Molloy and Reed \cite{MR1438613} gave the first general bound by using probabilistic techniques, showing that if $\Delta$ is large enough, then every graph with maximum degree $\Delta$ has a strong edge coloring with $1.998\Delta^{2}$ colors. Recently, the upper bound was improved to $1.93\Delta^{2}$ by Bruhn and Joos \cite{2015arXiv150402583B}. Very recently, this was further improved to $1.835\Delta^{2}$ by Bonamy, Perrett and Postle \cite{Bonamy}.

For planar graphs, Faudree, Schelp, Gy{\'a}rf{\'a}s and Tuza \cite{MR1412876} established the following upper bound using the Four Color Theorem and Vizing's Theorem. 
\begin{theorem}[Faudree \etal \cite{MR1412876}]
Every planar graph with maximum degree $\Delta$ has a strong edge coloring with $4\Delta + 4$ colors. 
\end{theorem}

Kostochka \cite{MR3398865} showed that every subcubic planar multigraph without loops (multiple edges are allowed) has a strong edge coloring with nine colors. Borodin and Ivanova \cite{MR3117054} showed that if $G$ is a planar graph with girth at least $40 \lfloor \frac{\Delta}{2} \rfloor+ 1$, then $\chiup_{s}'(G) \leq 2\Delta -1$. Chang, Montassier, P{\^e}cher and Raspaud \cite{MR3268687} improved the lower bound on the girth as the following.

\begin{theorem}[Chang \etal \cite{MR3268687}]
Let $\Delta \geq 4$ be an integer. If $G$ is a planar graph with maximum degree at most $\Delta$ and girth at least $10\Delta + 46$, then $\chiup_{s}'(G) \leq 2\Delta - 1$. 
\end{theorem}

In this note, we improve the lower bound on the girth to $10\Delta - 4$ when $\Delta \geq 4$, by following and improving the method in \cite{MR3268687}. In addition, we give two results with the condition on the maximum average degree. The maximum average degree $\mad(G)$ of a graph $G$ is defined as 
\[
\mad(G) = \max_{\emptyset \neq H \subseteq G}\left\{\frac{2|E(H)|}{|V(H)|}\right\}. 
\]

\section{Odd graphs}
The odd graph $O_{n}$ has one vertex for each of the $(n-1)$-element subsets of a $(2n-1)$-element set $\{1, 2, \dots, 2n-1\}$. Two vertices are adjacent if and only if the corresponding subsets are disjoint. $O_{n}$ is a Kneser graph $KG(2n-1,n-1)$. We use $v$ to denote a vertex in the odd graph, and use $[v]$ to denote the corresponding subset. By the definition, every edge $uv$ can be labeled with a unique element which is not contained in $[u]$ or $[v]$. By the particular structure of odd graphs, the labels on the edges induce a strong edge coloring of the odd graph. We will use this property in the proof of \autoref{MR}. A {\em walk} is a sequence of edges such that every two consecutive edges are adjacent. A {\em special walk} is a walk whose every two consecutive edges are distinct.  A {\em $k$-special walk} is a special walk of length $k$. A {\em $k$-path} is a path with $k$ edges. 

\subsection{Structural results}

\begin{lemma}
Let $O_{n}$ be the odd graph. Then the following items hold: 
\begin{enumerate}[label = (\alph*)]
\item \label{2-Path}
If the distance between $u$ and $v$ is two, then $[u]$ and $[v]$ have exactly one different element. Thus, for two vertices $u$ and $v$, if $[u]$ and $[v]$ have exactly $k$ different elements, then they can be reached from each other in $2k$ steps and this path is a shortest even path. 
\item \label{Replacing}
If $w_{1}w_{2}w_{3}$ is a path with $w_{1}w_{2}$ and $w_{2}w_{3}$ being labeled with $x$ and $y$ respectively, then $[w_{3}]$ is obtained from $[w_{1}]$ by replacing $y$ with $x$. 
\item Every $3$-path $v_{1}v_{2}v_{3}v_{4}$ is contained in a $6$-cycle. 
\end{enumerate}
\end{lemma}
\begin{proof}%
The first two items are trivial, so we do not include their proofs. Note that the sets corresponding to $v_{1}$ and $v_{3}$ differ only on one element. So we may assume that  $[v_{1}] = X \cup \{x_{1}\}$ and $[v_{3}] = X \cup \{x_{2}\}$, where $X$ is an $(n-2)$-subset of $\{1, 2, \dots, 2n-1\}$. Thus, we can obtain $[v_{2}]=\{1, 2, \dots, 2n-1\}\setminus (X \cup \{x_{1}, x_{2}\})$. Without loss of generality, we may assume that $[v_{4}]=\{1, 2, \dots, 2n-1\}\setminus (X \cup \{x_{2}, x_{3}\})$ and $x_{1} \neq x_{3}$. It is easy to check that the path $v_{1}v_{2}v_{3}v_{4}$ is contained in a $6$-cycle $v_{1}v_{2}v_{3}v_{4}v_{5}v_{6}$, where $[v_{5}] = X \cup \{x_{3}\}$ and $[v_{6}] = \{1, 2, \dots, 2n-1\}\setminus (X \cup \{x_{1}, x_{3}\})$. 
\end{proof}
In the following, we use $\undergroup{x, y}$ to denote the labels $x$ and $y$ on two consecutive edges on a walk.  

\begin{theorem}\label{2N}
Given two vertices $w_{0}, w_{2n}$ (not necessarily distinct) in $O_{n}$ ($n \geq 4$) with $\lambda_{1} \notin [w_{0}]$ and $\lambda_{2} \notin [w_{2n}]$, there exists a $2n$-special walk from $w_{0}$ to $w_{2n}$ such that its first edge is labeled with $\lambda_{1}$ and the last edge is labeled with $\lambda_{2}$.
\end{theorem}
\begin{proof}%
We may assume that $[w_{0}] = \{x_{1}, x_{2}, \dots, x_{k}\} \cup \{y_{k+1}, y_{k+2}, \dots, y_{n-1}\}$, $[w_{2n}] = \{x_{1}, x_{2}, \dots, x_{k}\} \cup \{z_{k+1}, z_{k+2}, \dots, z_{n-1}\}$ and $[w_{0}] \cup [w_{2n}] \cup \{s_{1}, s_{2}, \dots, s_{k+1}\} = \{1, \dots, 2n-1\}$, where $k \in \{0, 1, \dots, n - 1\}$. By symmetry, we divide the proof into three cases. 

\begin{case}
$\lambda_{1} \in \{z_{k+1}, z_{k+2}, \dots, z_{n-1}\}$ and $\lambda_{2} \in \{y_{k+1}, y_{k+2}, \dots, y_{n-1}\}$.
\end{case}
It is obvious that $k \neq n-1$. By symmetry, we may assume that $\lambda_{1} = z_{k+1}$ and $\lambda_{2} = y_{n-1}$.

\begin{subcase}
$k \leq n-4$.
\end{subcase}
In the odd graph $O_{n}$, we can find a special walk from $w_{0}$ to $w_{2n}$, whose edges are labeled with $\undergroup{z_{k+1}, y_{k+1}}$, $\undergroup{z_{k+2}, y_{k+2}}$, $\dots, \undergroup{z_{n-3}, y_{n-3}}$, $\undergroup{s_{1}, y_{n-2}}$, $\undergroup{s_{2}, s_{1}}, \dots, \undergroup{s_{k+1}, s_{k}}$, $\undergroup{z_{n-2}, s_{k+1}}$, $\undergroup{z_{n-1}, y_{n-1}}$.

\begin{subcase}
$k = n-3$.
\end{subcase}
In the odd graph $O_{n}$, we can find a special walk from $w_{0}$ to $w_{2n}$, whose edges are labeled with $\undergroup{z_{n-2}, x_{1}}$, $\undergroup{s_{1}, y_{n-2}}$, $\undergroup{s_{2}, s_{1}}$, $\undergroup{s_{3}, s_{2}}, \dots, \undergroup{s_{k}, s_{k-1}}$, $\undergroup{x_{1}, s_{k}}$, $\undergroup{z_{n-1}, y_{n-1}}$. (It needs $n$ at least $4$, otherwise $x_{1}$ does not exist.)

\begin{subcase}
$k = n-2$.
\end{subcase}
In the odd graph $O_{n}$, we can find a special walk from $w_{0}$ to $w_{2n}$, whose edges are labeled with $\undergroup{z_{n-1}, y_{n-1}}$, $\undergroup{s_{1}, x_{1}}$, $\undergroup{s_{2}, s_{1}}$, $\undergroup{s_{3}, s_{2}}, \dots, \undergroup{s_{n-3}, s_{n-4}}$, $\undergroup{y_{n-1}, s_{n-3}}$, $\undergroup{x_{1}, y_{n-1}}$. 

\begin{case}
$\lambda_{1} \in \{z_{k+1}, z_{k+2}, \dots, z_{n-1}\}$ and $\lambda_{2} \in \{s_{1}, s_{2}, \dots, s_{k+1}\}$.
\end{case}
It is obvious that $k \neq n-1$. By symmetry, we may assume that $\lambda_{1} = z_{k+1}$ and $\lambda_{2} = s_{k+1}$.

If $k = 0$, then we can find a special walk from $w_{0}$ to $w_{2n}$, whose edges are labeled with $\undergroup{z_{1}, y_{1}}$, $\undergroup{z_{2}, y_{2}}, \dots, \undergroup{z_{n-2}, y_{n-2}}$, $\undergroup{s_{1}, y_{n-1}}$, $\undergroup{z_{n-1}, s_{1}}$.

If $k \in \{1, 2, \dots, n-2\}$, then we can find a special walk from $w_{0}$ to $w_{2n}$, whose edges are labeled with $\undergroup{z_{k+1}, y_{k+1}}$, $\undergroup{z_{k+2}, y_{k+2}}, \dots, \undergroup{z_{n-2}, y_{n-2}}$, $\undergroup{s_{1}, y_{n-1}}$, $\undergroup{s_{2}, s_{1}}$, $\undergroup{s_{3}, s_{2}}, \dots, \undergroup{s_{k+1}, s_{k}}$, $\undergroup{z_{n-1}, s_{k+1}}$.

\begin{case}
$\lambda_{1} \in \{s_{1}, s_{2}, \dots, s_{k+1}\}$ and $\lambda_{2} \in \{s_{1}, s_{2}, \dots, s_{k+1}\}$.
\end{case}
\begin{subcase}
$\lambda_{1} = \lambda_{2} = s_{1}$.
\end{subcase}

If $k=0$, then we can find a special walk whose edges are labeled with $\undergroup{s_{1}, y_{1}}$, $\undergroup{z_{2}, y_{2}}$, $\undergroup{z_{3}, y_{3}}, \dots, \undergroup{z_{n-1}, y_{n-1}}$, $\undergroup{z_{1}, s_{1}}$.

If $k=1$, then we can find a special walk whose edges are labeled with $\undergroup{s_{1}, x_{1}}$, $\undergroup{z_{2}, y_{2}}$, $\undergroup{z_{3}, y_{3}}, \dots, \undergroup{z_{n-1}, y_{n-1}}$, $\undergroup{x_{1}, s_{1}}$.

If $k=2$, then we can find a special walk whose edges are labeled with $\undergroup{s_{1}, x_{1}}$, $\undergroup{s_{3}, y_{3}}$, $\undergroup{z_{3}, s_{3}}$, $\undergroup{z_{4}, y_{4}}, \dots, \undergroup{z_{n-1}, y_{n-1}}$, $\undergroup{x_{1}, s_{1}}$.

If $k \in \{3, \dots, n-2\}$, then we can find a special walk whose edges are labeled with $\undergroup{s_{1}, x_{1}}$, $\undergroup{s_{3}, y_{k+1}}$, $\undergroup{s_{4}, s_{3}}$, $\undergroup{s_{5}, s_{4}}, \dots, \undergroup{s_{k+1}, s_{k}}$, $\undergroup{z_{k+1}, s_{k+1}}$, $\undergroup{z_{k+2}, y_{k+2}}, \dots, \undergroup{z_{n-1}, y_{n-1}}$, $\undergroup{x_{1}, s_{1}}$.

If $k = n-1$, then we can find a special walk whose edges are labeled with $s_{1}, \undergroup{x_{2}, s_{2}}, \undergroup{x_{3}, x_{2}}, \undergroup{x_{4}, x_{3}}, \dots, \undergroup{x_{k}, x_{k-1}}, \undergroup{s_{2}, x_{k}}$, $s_{1}$. 

\begin{subcase}
$\lambda_{1} \neq \lambda_{2}$. We may assume that $\lambda_{1} = s_{1}$ and $\lambda_{2} = s_{k+1}$.
\end{subcase}

If $k=1$, then we can find a special walk whose edges are labeled with $\undergroup{s_{1}, y_{2}}$, $\undergroup{s_{2}, s_{1}}$, $\undergroup{z_{3}, y_{3}}$, $\undergroup{z_{4}, y_{4}}$, $\undergroup{z_{5}, y_{5}}, \dots, \undergroup{z_{n-1}, y_{n-1}}$, $\undergroup{z_{2}, s_{2}}$.

If $k=2$, then we can find a special walk whose edges are labeled with $\undergroup{s_{1}, y_{3}}$, $\undergroup{s_{2}, s_{1}}$, $\undergroup{s_{3}, s_{2}}$, $\undergroup{z_{4}, y_{4}}$, $\undergroup{z_{5}, y_{5}}, \dots, \undergroup{z_{n-1}, y_{n-1}}$, $\undergroup{z_{3}, s_{3}}$.

If $k \in \{3, 4, \dots, n-2\}$, then we can find a special walk whose edges are labeled with $\undergroup{s_{1}, y_{k+1}}$, $\undergroup{s_{2}, s_{1}}$, $\undergroup{s_{3}, s_{2}}$, $\undergroup{s_{4}, s_{3}}$, $\undergroup{s_{5}, s_{4}}, \dots, \undergroup{s_{k+1}, s_{k}}$, $\undergroup{z_{k+2}, y_{k+2}}$, $\undergroup{z_{k+3}, y_{k+3}}, \dots, \undergroup{z_{n-1}, y_{n-1}}$, $\undergroup{z_{k+1}, s_{k+1}}$.

If $k=n-1$, then we can find a special walk whose edges are labeled with $\undergroup{s_{1}, x_{n-1}}$, $\undergroup{s_{3}, s_{1}}$, $\undergroup{s_{4}, s_{3}}$, $\undergroup{s_{5}, s_{4}}$, $\undergroup{s_{6}, s_{5}}, \dots, \undergroup{s_{k}, s_{k-1}}$, $\undergroup{s_{k+1}, s_{k}}$, $\undergroup{x_{n-1}, s_{k+1}}$. This completes the proof of the theorem. 
\resetcounter
\end{proof}

Note that we restrict $n \geq 4$ in the above theorem since the result is not true for $n = 3$, see the next subsection. For $n= 3$, we have the following result. 

\begin{theorem}\label{2N+3}
Given two vertices $w_{0}, w_{9}$ (not necessarily distinct) in $O_{3}$ with $\lambda_{1} \notin [w_{0}]$ and $\lambda_{2} \notin [w_{9}]$, there exists a $9$-special walk from $w_{0}$ to $w_{9}$ such that its first edge is labeled with $\lambda_{1}$ and the last edge is labeled with $\lambda_{2}$.
\end{theorem}
\begin{proof}
By symmetry, we divide the proof into three cases in terms of the size of intersection of $[w_{0}]$ and $[w_{9}]$, and each case is divided into some subcases. 

\begin{case}
Assume that $[w_{0}] = \{y_{1}, y_{2}\}$, $[w_{9}] = \{z_{1}, z_{2}\}$ and $\{y_{1}, y_{2}, z_{1}, z_{2}, s_{1}\} = \{1, 2, 3, 4, 5\}$. 
\end{case}

When $\lambda_{1} = z_{1}$ and $\lambda_{2} = y_{1}$, we can find a special walk whose edges are labeled with $\undergroup{z_{1}, y_{1}}, \undergroup{z_{2}, y_{2}}, \undergroup{s_{1}, z_{1}}, \undergroup{y_{2}, z_{2}}, y_{1}$. 

When $\lambda_{1} = z_{1}$ and $\lambda_{2} = s_{1}$, we can find a special walk whose edges are labeled with $\undergroup{z_{1}, y_{1}}, \undergroup{z_{2}, y_{2}}, \undergroup{y_{1}, z_{1}}, \undergroup{y_{2}, z_{2}}, s_{1}$.

When $\lambda_{1} = \lambda_{2} = s_{1}$, we can find a special walk whose edges are labeled with $\undergroup{s_{1}, y_{1}}, \undergroup{z_{2}, y_{2}}, \undergroup{y_{1}, s_{1}}, \undergroup{y_{2}, z_{2}}, s_{1}$.

\begin{case}
Assume that $[w_{0}] = \{x_{1}, y_{2}\}$, $[w_{9}] = \{x_{1}, z_{2}\}$ and $\{x_{1}, y_{2}, z_{2}, s_{1}, s_{2}\} = \{1, 2, 3, 4, 5\}$.
\end{case}

When $\lambda_{1} = z_{2}$ and $\lambda_{2} = y_{2}$, we can find a special walk whose edges are labeled with  $\undergroup{z_{2}, x_{1}}, \undergroup{s_{2}, z_{2}}, \undergroup{x_{1}, y_{2}}, \undergroup{s_{1}, x_{1}}, y_{2}$. 

When $\lambda_{1} = z_{2}$ and $\lambda_{2} = s_{2}$, we can find a special walk whose edges are labeled with $\undergroup{z_{2}, y_{2}}, \undergroup{s_{2}, x_{1}}, \undergroup{y_{2}, s_{2}}, \undergroup{s_{1}, z_{2}}, s_{2}$. 

When $\lambda_{1} =s_{1}$ and $\lambda_{2} = s_{2}$, we can find a special walk whose edges are labeled with $\undergroup{s_{1}, x_{1}}, \undergroup{s_{2}, s_{1}}, \undergroup{x_{1}, s_{2}}, \undergroup{s_{1}, x_{1}}, s_{2}$. 

When $\lambda_{1} = \lambda_{2} = s_{1}$, we can find a special walk whose edges are labeled with $\undergroup{s_{1}, x_{1}}, \undergroup{z_{2}, y_{2}}, \undergroup{s_{2}, s_{1}}, \undergroup{y_{2}, z_{2}}, s_{1}$. 

\begin{case}
Assume that $[w_{0}] = [w_{9}] = \{x_{1}, x_{2}\}$ and $\{x_{1}, x_{2}, s_{1}, s_{2}, s_{3}\} = \{1, 2, 3, 4, 5\}$. 
\end{case}

When $\lambda_{1} =s_{1}$ and $\lambda_{2} = s_{3}$, we can find a special walk whose edges are labeled with $\undergroup{s_{1}, x_{1}}, \undergroup{s_{3}, x_{2}}, \undergroup{x_{1}, s_{3}}, \undergroup{s_{2}, x_{1}}, s_{3}$. 

When $\lambda_{1} = \lambda_{2} = s_{3}$, we can find a special walk whose edges are labeled with $\undergroup{s_{3}, x_{1}}, \undergroup{s_{2}, x_{2}}, \undergroup{x_{1}, s_{3}}, \undergroup{s_{1}, x_{1}}, s_{3}$. This completes the proof of the theorem. 
\end{proof}

By induction on the length of the special walk, we can conclude the following results. 

\begin{theorem}\label{2N+}
Given two vertices $w_{0}, w_{\ell+1}$ (not necessarily distinct) in $O_{n}$ ($n \geq 4$) with $\lambda_{1} \notin [w_{0}]$ and $\lambda_{2} \notin [w_{\ell+1}]$, there exists an $(\ell +1)$-special walk from $w_{0}$ to $w_{\ell+1}$ such that its first edge is labeled with $\lambda_{1}$ and the last edge is labeled with $\lambda_{2}$, where $\ell$ is an arbitrary integer at least $2n-1$.
\end{theorem}

\begin{theorem}\label{2N3+}
Given two vertices $w_{0}, w_{\ell+1}$ (not necessarily distinct) in $O_{3}$ with $\lambda_{1} \notin [w_{0}]$ and $\lambda_{2} \notin [w_{\ell+1}]$, there exists an $(\ell + 1)$-special walk from $w_{0}$ to $w_{\ell+1}$ such that its first edge is labeled with $\lambda_{1}$ and the last edge is labeled with $\lambda_{2}$, where $\ell$ is an arbitrary integer at least $8$.
\end{theorem}

In the above theorems, we require the first and the last edges to be prescribed, so we immediately have the following corollaries.  

\begin{corollary}\label{A2N+}
Given two vertices $w_{1}, w_{\ell}$ (not necessarily distinct) in $O_{n}$ ($n \geq 4$) with $\lambda_{1} \notin [w_{1}]$ and $\lambda_{2} \notin [w_{\ell}]$, there exists an $(\ell - 1)$-special walk from $w_{1}$ to $w_{\ell}$ such that its first edge is not labeled with $\lambda_{1}$ and the last edge is not labeled with $\lambda_{2}$, where $\ell$ is an arbitrary integer at least $2n-1$. 
\end{corollary}

\begin{corollary}\label{A2N3+}
Given two vertices $w_{1}, w_{\ell}$ (not necessarily distinct) in $O_{3}$ with $\lambda_{1} \notin [w_{1}]$ and $\lambda_{2} \notin [w_{\ell}]$, there exists an $(\ell - 1)$-special walk from $w_{1}$ to $w_{\ell}$ such that its first edge is not labeled with $\lambda_{1}$ and the last edge is not labeled with $\lambda_{2}$, where $\ell$ is an arbitrary integer at least $8$.
\end{corollary}

\subsection{Sharpness of the length in \autoref{2N} and \autoref{2N+3}}
In this subsection, we show that the length of the special walks in \autoref{2N} and \autoref{2N+3} can not be can not be reduced anymore. 

\begin{theorem}
The length of the special walk in \autoref{2N} is shortest possible. 
\end{theorem}
\begin{proof}
Let $[w_{0}] = \{y_{1}, y_{2}, \dots, y_{n-1}\}$, $[w] = \{z_{1}, z_{2}, \dots, z_{n-1}\}$ and $\lambda_{1} = s_{1}$. By Lemma 1\ref{2-Path}, if we go from $w_{0}$ to $w$ such that the first edge is labeled with $\lambda_{1}$, then the length of a shortest even path is at least $2n$. Hence, the length of the special walk in \autoref{2N} cannot be improved to a smaller even number. 

And if $[w_{2n}] = \{y_{1}, y_{2}, \dots, y_{n-1}\}$, we go from $w_{0}$ to $w_{2n}$ such that both the first edge and the last edge are labeled with $s_{1}$, in other words, we go from $w_{0}$ to $w$ such that the first edge is labeled with $s_{1}$, and then go through the edge $ww_{2n}$, then the length of a  shortest odd special walk from $w_{0}$ to $w_{2n}$ is at least $2n + 1$.  Hence, the length of the special walk in \autoref{2N} cannot be improved to a smaller odd number. 
\end{proof}

\begin{theorem}
The length of the special walk in \autoref{2N+3} is shortest possible. 
\end{theorem}
\begin{proof}
Firstly, we show that the length $6$ is impossible. For example, let $[w_{0}] = [u_{0}] = \{1, 2\}$ and $[w_{1}] = [u_{1}] = \{3, 4\}$, if we want to go from $w_{0}$ to $u_{0}$ such that the first edge is $w_{0}w_{1}$ and the last edge is $u_{1}u_{0}$, then we should go from $w_{1}$ to $u_{1}$ by a special walk of length $4$, a contradiction. 

Secondly, we show that the length $7$ is impossible. For example, let $[w_{0}] = [w_{7}] = \{1, 2\}$, $[w_{1}] = \{3, 4\}$ and $[w_{6}] = \{4, 5\}$. Suppose that there exists a $7$-special walk $w_{0}w_{1}w_{2}w_{3}w_{4}w_{5}w_{6}w_{7}$. Note that $w_{1}w_{2}$ can be labeled with $1$ or $2$, we may assume that it is labeled with $1$, thus $[w_{2}] = \{2, 5\}$. We may assume that $w_{2}w_{3}, w_{3}w_{4}, w_{4}w_{5}, w_{5}w_{6}$ are labeled with $\lambda_{3}, \lambda_{4}, \lambda_{5}, \lambda_{6}$, respectively. If $5 \notin [w_{4}]$, then we should go along $w_{4}w_{3}w_{2}$ such that $[w_{2}] = \{2, 5\}$, and then $\lambda_{4} = 5$; similarly, we should go along $w_{4}w_{5}w_{6}$ such that $[w_{6}] = \{4, 5\}$, and again we have that $\lambda_{5} = 5$, a contradiction. So we have that $[w_{4}] = \{5, \lambda\}$. Notice that $\lambda_{3} \in [w_{4}]$ and $\lambda_{6} \in [w_{4}]$, thus $\lambda = \lambda_{3} = \lambda_{6}$, but this is impossible since $\lambda_{3} \in \{3, 4\}$ and $\lambda_{6} \in \{1, 2\}$. 

Thirdly, we show that the length $8$ is impossible. For example, let $[w_{0}] = [w_{7}] = \{1, 2\}$ and $[w_{1}] = [w_{8}] = \{3, 4\}$.  Suppose that there exists a $8$-special walk $w_{0}w_{1}w_{2}w_{3}w_{4}w_{5}w_{6}w_{7}w_{8}$ and $w_{1}w_{2}, w_{2}w_{3}, w_{3}w_{4}, w_{4}w_{5}, w_{5}w_{6}, w_{6}w_{7}$ are labeled with $\lambda_{2}, \lambda_{3}, \lambda_{4}, \lambda_{5}, \lambda_{6}, \lambda_{7}$ respectively. By Lemma 1\ref{Replacing}, the three elements $\lambda_{2}, \lambda_{4}, \lambda_{6}$ are distinct and $\{1, 2\} \subseteq \{\lambda_{2}, \lambda_{4}, \lambda_{6}\}$. So we may assume that $\{\lambda_{2}, \lambda_{4}, \lambda_{6}\} = \{1, 2, \lambda\}$. Note that both $w_{0}w_{1}$ and $w_{7}w_{8}$ are labeled with $5$, thus $\lambda \neq 5$ and $\lambda \in \{3, 4\}$.  Note that $\lambda_{2} \notin \{3, 4\}$, Lemma 1\ref{Replacing} implies that $\lambda = \lambda_{4} \in \{3, 4\}$. By the construction of the odd graph, we have that $\lambda \notin [w_{3}]$, but $[w_{3}]$ is obtained from $\{3, 4\}$ by replacing $\lambda_{3}$ with $\lambda_{2}$, and then $\lambda_{3} = \lambda = \lambda_{4}$, a contradiction. 
\end{proof}

\section{Applications to strong edge coloring}
An $\ell$-thread in a graph $G$ is a path of length $\ell + 1$ in $G$ whose $\ell$ internal vertices have degree $2$ in the full graph $G$. Note that the endpoints of the $\ell$-thread can be the same vertex. For example, an $(\ell +1)$-cycle contains many $\ell$-threads.

Let $R$ be a subgraph of $G$. We call $R$ a {\em $\kappa$-reducible configuration} if there exists an induced subgraph $R'$ of $R$ such that each strong $\kappa$-edge-coloring of $G - R'$ can be extended to a strong $\kappa$-edge-coloring of $G$. 

Now, we address the following result for reducible configurations. 
\begin{lemma}\label{RC}
Let $G$ be a graph with maximum degree at most $\kappa$.
\begin{enumerate}
\item If $\kappa \geq 4$, then $(2\kappa - 1)^{+}$-caterpillar is $(2\kappa - 1)$-reducible. Moreover, $(2\kappa - 2)$-caterpillar is not $(2\kappa - 1)$-reducible. 
\item If $\kappa = 3$, then $8^{+}$-caterpillar is $5$-reducible. Moreover, $7$-caterpillar is not $5$-reducible. 
\end{enumerate}
\end{lemma}
\begin{proof}
Suppose that $G$ has a caterpillar tree with central path $u_{0}u_{1} u_{2}\dots u_{\ell}u_{\ell+1}$, where $\ell \geq 2\kappa - 1$ if $\kappa \geq 4$, and $\ell \geq 8$ if $\kappa =3$. That is, every vertex $u_{i}$ with $1 \leq i \leq \ell$ is adjacent to at least $\deg(u_{i}) - 2$ pendent vertices. Let $R$ be the graph obtained from $G - \{u_{2}, u_{3}, \dots, u_{\ell -1}\}$ by deleting all the pendent vertices adjacent to $u_{1}$ and $u_{\ell}$, except $u_{0}$ and $u_{\ell + 1}$. For any arbitrary strong edge coloring $\varphi$ of $R$ using colors from $\{1, 2, \dots, 2\kappa-1\}$, we may assume that the color set for the edges incident with $u_{0}$  and $u_{\ell + 1}$ is a subset of a $\kappa$-set $A$ and a $\kappa$-set $B$, respectively. By \autoref{2N+} and \ref{2N3+}, we can find a special walk $w_{0}w_{1} \dots w_{\ell}w_{\ell + 1}$ in the odd graph $O_{\kappa}$ such that it satisfies the following properties: 
\begin{enumerate}[label = (\alph*)]
\item its length is exactly $\ell + 1$;
\item $[w_{0}] = \overline{A}$ and $[w_{\ell + 1}] = \overline{B}$, where $\overline{A}$ and $\overline{B}$ are the complements of $A$ and $B$ respectively. 
\item $w_{0}w_{1}$ is labeled with $\phi(u_{0}u_{1})$ and $w_{\ell}w_{\ell + 1}$ is labeled with $\phi(u_{\ell}u_{\ell+1})$. 
\end{enumerate}

Hence, we can color the edges incident to $u_{i}$ with colors in $\overline{[w_{i}]}$, and the edge $u_{i}u_{i+1}$ is colored with the label on $w_{i}w_{i+1}$ for $1 \leq i \leq \ell$. Finally, we obtain a strong edge coloring of $G$ using colors from $\{1, 2, \dots, 2\kappa -1\}$, a contradiction. 
\end{proof}

Before we prove the main theorems, we need the following result due to Ne{\v{s}}et{\v{r}}il, Raspaud and Sopena \cite{MR1439297}. 

\begin{theorem}[Ne{\v{s}}et{\v{r}}il, Raspaud and Sopena \cite{MR1439297}]\label{thread}
If $G$ is a planar graph with girth at least $5\ell+1$, then $G$ contains an $\ell$-thread or a vertex with degree at most one. 
\end{theorem}

\begin{theorem}\label{MR}
If $G$ is a planar graph with maximum degree at most $\Delta$ and girth at least $10\Delta - 4$, where $\Delta \geq 4$, then $G$ has a strong edge coloring with $2\Delta - 1$ colors. 
\end{theorem}
\begin{proof}
Let $G$ be a counterexample to the theorem with minimum number of vertices. It is obvious that $G$ is connected and the minimum degree is at least one. Let $A$ be the set of pendent vertices in $G$, and let $H = G - A$. Note that $H$ is connected and $\delta(H) \geq 1$, for otherwise $G$ is a star and $\chiup_{s}'(G) = \Delta(G) < 2\Delta - 1$. 

Suppose that $uv$ is an edge in $H$ with $\deg_{H}(u) = 1$. Thus, the vertex $u$ is adjacent to a pendent vertex $w$ in $G$. By the minimality of $G$, the graph $G - w$ has a strong edge coloring with $2\Delta - 1$ colors. Note that $uw$ has at most $ (\deg_{G}(u) - 1) + \deg_{G}(v) \leq 2\Delta - 1$ edges within distance two (including $uw$ itself), so we can extend the strong edge coloring of $G - w$ to the whole graph $G$. Thus, the minimum degree of $H$ is at least two.

By \autoref{thread}, the graph $H$ has a $(2\Delta - 1)$-thread. Correspondingly, the original graph $G$ has a $(2\Delta - 1)$-caterpillar tree, this is a contradiction due to \autoref{RC}. 
\end{proof}

Using the same technique, we can prove the following result which has been proved by Borodin and Ivanova \cite{MR3117054}. 
\begin{theorem}[Borodin and Ivanova \cite{MR3117054}]
If $G$ is a subcubic planar graph with girth at least $41$, then $G$ has a strong edge coloring with $5$ colors. 
\end{theorem}

Next, we give some results with bounded maximum average degree. 

\begin{theorem}[Cranston and West \cite{MR3603558}]\label{Thread}
If $G$ has average degree less than $2 + \frac{2}{3\ell - 1}$ and $\Delta(G) \geq 3$, then $G$ contains an $\ell$-thread or a vertex with degree at most one. 
\end{theorem}

Chang, Montassier, P{\^e}cher and Raspaud \cite{MR3268687} gave the following result for closed special walks.

\begin{theorem}[Chang \etal \cite{MR3268687}]\label{CMPR}
Every vertex  in $O_{\Delta}$ is contained in a closed $2\ell$-special walk for arbitrary integers $\Delta, \ell \geq 3$. 
\end{theorem}

Let $C_{\kappa, \Delta}$ be the graph obtained from $C_{\kappa}$ by adding $(\Delta - 2)$ pendant edges to each vertex in $C_{\kappa}$. We can easily address the following result. 
\begin{theorem}\label{CW}
\mbox{}
\begin{enumerate}[label = (\alph*)]
\item\label{3_6_a} If $\kappa$ is even and at least $6$, or $\kappa$ is odd and at least $2\Delta - 1$ with $\Delta \geq 4$, then $\chiup_{s}'(C_{\kappa, \Delta}) = 2\Delta - 1$. 
\item\label{3_6_b} If $\kappa \geq 5$ and $\kappa \neq 7$, then $\chiup_{s}'(C_{\kappa, 3}) = 2\Delta - 1 = 5$.
\end{enumerate}
\end{theorem}
\begin{proof}
Suppose that the unique cycle in $C_{\kappa, \Delta}$ is $v_{0}v_{1}\dots v_{\kappa}$. Note that the even-girth of $O_{\Delta}$ is six and the odd-girth is $2\Delta - 1$. If $\kappa$ is even, then there is a closed special $\kappa$-walk in $O_{\Delta}$ by \autoref{CMPR}. If $\kappa = 2\Delta - 1$, then each smallest odd cycle in $O_{\Delta}$ is a closed special $\kappa$-walk. If $\kappa$ is odd and at least $2\Delta + 3$, then there is a closed special $\kappa$-walk in $O_{\Delta}$ by \autoref{2N+} and \ref{2N3+}. If $\kappa = 2\Delta + 1 \geq 9$, then there is a closed special $\kappa$-walk in $O_{\Delta}$ by \autoref{2N+}. Hence, the odd graph $O_{\Delta}$ has a closed special $\kappa$-walk $w_{0}w_{1}\dots w_{\kappa}$ as desired. Similar to the previous, we can color the edges incident to $v_{i}$ with colors in $[w_{i}]$, and the edge $v_{i}v_{i+1}$ is colored with the label on $w_{i}w_{i+1}$ for $0 \leq i \leq \kappa$. 
\end{proof}

\begin{theorem}
Let $G$ be a graph with maximum degree at most $\Delta$ ($\Delta \geq 4$). If $\evengirth(G) \geq 6$, $\oddgirth(G) \geq 2\Delta - 1$ and $\mad(G) < 2 + \frac{1}{3\Delta - 2}$, then $\chiup_{s}'(G) \leq 2\Delta - 1$. 
\end{theorem}
\begin{proof}
Let $G$ be a counterexample to the theorem with minimum number of vertices. It is obvious that $G$ is connected and the minimum degree is at least one. Let $A$ be the set of pendent vertices in $G$, and let $H = G - A$. It is easy to show that the minimum degree of $H$ is at least two. 

If $H$ has a $(2\Delta - 1)$-thread, then $G$ has a $(2\Delta - 1)$-caterpillar tree, which contradicts \autoref{RC}. So we may assume that there exists no $(2\Delta - 1)$-threads in $H$. Note that $H$ is connected. By \autoref{Thread} and the hypothesis, $H$ is a cycle and $G$ is a subgraph of $C_{\kappa, \Delta}$, this is a contradiction due to \autoref{CW}\ref{3_6_a}. 
\end{proof}

Similarly, we have the following result for subcubic graphs. 
\begin{theorem}
If $G$ is a subcubic graph with girth at least $8$ and $\mad(G) < 2 + \frac{2}{23}$, then $\chiup_{s}'(G) \leq 5$. 
\end{theorem}

\begin{remark}
The essential of all the proofs is finding a special walk with fixed length between any two vertices in the odd graph, in which the first edge and the last edge are specified. Obviously, all the results cannot be directly extended to the list version of strong edge coloring. 
\end{remark}

\begin{remark}
As far as we know, some results in this note have been used and/or improved in \cite{2015arXiv150803052J, DeOrsey2015}.
\end{remark}

\vskip 3mm \vspace{0.3cm} \noindent{\bf Acknowledgments.} This project was supported by the National Natural Science Foundation of China (11101125) and partially supported by the Fundamental Research Funds for Universities in Henan (YQPY20140051). 

The authors would like to thank the anonymous reviewers for their valuable comments on earlier version.

\end{document}